\documentclass[11pt]{article}

\usepackage[letterpaper,margin=1in]{geometry}
\usepackage{amsmath,amssymb,amsthm,mathtools}
\usepackage{enumitem}
\usepackage{xcolor}
\usepackage{microtype}
\usepackage{float}
\usepackage{tikz}
\usetikzlibrary{arrows.meta,calc,positioning}
\usepackage[most]{tcolorbox}
\usepackage[hidelinks]{hyperref}

\definecolor{linkblue}{RGB}{24,76,145}
\definecolor{setblue}{RGB}{35,91,148}
\definecolor{setfill}{RGB}{225,238,249}
\definecolor{repairorange}{RGB}{213,119,36}
\definecolor{querygreen}{RGB}{46,128,91}
\hypersetup{colorlinks=true,linkcolor=linkblue,citecolor=linkblue,urlcolor=linkblue}
\setlist[itemize]{leftmargin=1.5em,itemsep=0.2em,topsep=0.35em}
\setlist[enumerate]{leftmargin=1.7em,itemsep=0.25em,topsep=0.35em}

\newtheorem{theorem}{Theorem}[section]
\newtheorem{lemma}[theorem]{Lemma}
\newtheorem{proposition}[theorem]{Proposition}
\newtheorem{corollary}[theorem]{Corollary}

\newcommand{\R}{\mathbb{R}}
\newcommand{\K}{\mathcal{K}}
\newcommand{\B}{\mathbb{B}}
\newcommand{\bd}{\partial}
\newcommand{\dist}{\operatorname{dist}}
\newcommand{\diam}{\operatorname{diam}}
\newcommand{\MEM}{\operatorname{MEM}_{\K}}
\newcommand{\ip}[2]{\langle #1,#2\rangle}
\newcommand{\norm}[1]{\lVert #1\rVert}
\newcommand{\eps}{\varepsilon}

\newtcolorbox{algorithmbox}[1]{
  enhanced,
  colback=setfill!45,
  colframe=setblue,
  boxrule=0.8pt,
  arc=1.5mm,
  left=1.5mm,right=1.5mm,top=1mm,bottom=1mm,
  title={#1},
  fonttitle=\bfseries,
  coltitle=black,
  attach boxed title to top left={xshift=2mm,yshift=-2mm},
  boxed title style={colback=white,colframe=setblue,boxrule=0.8pt,arc=1mm}
}

\title{A Linearly Convergent Projection-Free Algorithm\\for Smooth Convex Sets}
\author{Elad Hazan}
\date{}

\begin{document}
\maketitle

\begin{abstract}
We consider minimizing a smooth, strongly convex function over a convex set. Projected gradient descent is known to converge linearly in this setting, but each iteration requires a projection onto the feasible set, which may be computationally expensive.

We show that when the feasible set is smooth, projection can be replaced by one gradient computation and a single supporting-tangent computation per iteration, while preserving linear convergence. Moreover, the required tangent can be approximated to sufficient accuracy using $\widetilde O(d)$ membership-oracle queries, where $d$ is the ambient dimension. Previously, projection-free linear convergence was known only for polyhedral sets or for sets that are both smooth and strongly convex.

\end{abstract}

\section{Introduction}

We consider minimizing a smooth, strongly convex function over a convex body.
Projected gradient descent converges linearly in this setting, but projection
onto the feasible set may be computationally expensive and can dominate the
running time.  This motivates projection-free methods, beginning with the
Frank--Wolfe algorithm \cite{frankwolfe1956}, which replaces projection by
linear optimization over the feasible set.

The ordinary Frank-Wolfe algorithm has a sublinear convergence rate in
general.  For smooth, strongly convex objectives, linear convergence is known
over certain domains, including polytopes and sets that are both smooth and
strongly convex.

We show that projection-free linear convergence is also possible over smooth
convex sets $\K$, which may contain flat boundary pieces.  Such sets satisfy a
uniform interior rolling-ball condition: at every boundary point, a Euclidean
ball of radius $\rho$ can be placed tangent to the boundary and entirely
inside $\K$; see Figure~\ref{fig:smooth-set}.

The key observation is that a supporting halfspace gives a tractable
relaxation of the feasible set.  The rolling-ball condition lets us repair
the relaxed solution while moving it by only $O(s^2/\rho)$ inside a
neighborhood of radius $s$.  Hence one tangent implements the local
linear-optimization step to quadratic accuracy.

The following is our main guarantee.  Let $d\ge2$, let
$\K\subset\R^d$ be a compact $\rho$-smooth convex body, and let $f$ be
differentiable on a neighborhood of $\K$, with $\beta$-Lipschitz gradient
and $\alpha$-strong convexity on $\K$.  Let $G$ bound the gradient norm on
$\K$, and let $H_0$ bound the initial objective gap.  Our algorithm reaches
objective error at most $\eps$ after
\[
  O\!\left(
    \left(\frac{\beta}{\alpha}+\frac{G}{\alpha\rho}\right)
    \log\frac{H_0}{\eps}
  \right)
\]
iterations.  Each iteration uses one gradient computation and
$\widetilde O(d)$ membership queries.  The algorithm uses neither projection
onto $\K$ nor a local or global linear-optimization oracle.

\begin{figure}[H]
\centering
\begin{tikzpicture}[x=1.18cm,y=1.18cm,>=Latex,font=\small]
  \path[fill=setfill]
    (-2.15,-1) -- (2.15,-1)
    arc[start angle=-90,end angle=90,radius=1]
    -- (-2.15,1)
    arc[start angle=90,end angle=270,radius=1] -- cycle;
  \draw[setblue,very thick]
    (-2.15,-1) -- (2.15,-1)
    arc[start angle=-90,end angle=90,radius=1]
    -- (-2.15,1)
    arc[start angle=90,end angle=270,radius=1] -- cycle;
  \coordinate (q) at (2.969,0.574);
  \coordinate (m) at (2.477,0.230);
  \path[fill=repairorange!14] (m) circle[radius=0.60];
  \draw[repairorange,very thick] (m) circle[radius=0.60];
  \draw[gray!75,dashed,thick] (2.28,1.56) -- (3.66,-0.41)
    node[pos=0.07,above right,black] {supporting hyperplane};
  \draw[repairorange,thick,-{Latex[length=2.2mm]}] (m) -- (q)
    node[midway,above left=-1pt] {$\rho$};
  \draw[setblue,very thick,-{Latex[length=2.4mm]}] (q) -- (3.79,1.148)
    node[above right=-2pt] {$n$};
  \fill[black] (q) circle[radius=1.5pt] node[below right=2pt] {$q$};
  \fill[repairorange] (m) circle[radius=1.25pt];
  \node[setblue] at (-0.15,-0.10) {convex body $\K$};
\end{tikzpicture}
\caption{ \small{ A $\rho$-smooth body $\K$ has an interior tangent ball of radius
$\rho$ at every supporting pair $(q,n)$.   A useful example is a parallel body
$\K=C+\rho\B(0,1)$, where $C$ is compact and convex. Flat sides are allowed and show that the
rolling-ball condition is strictly weaker than strong convexity.}}
\label{fig:smooth-set}
\end{figure}
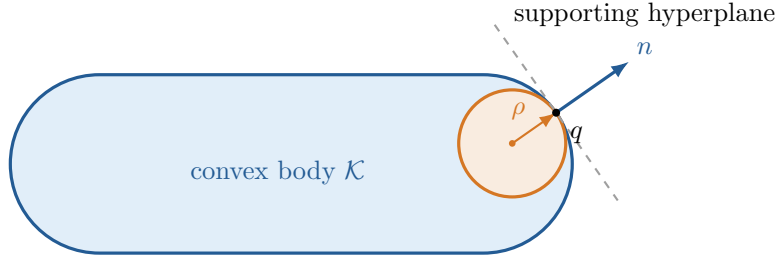

\subsection{Related work}

For an extensive review of the Frank Wolfe method and its extensions, see the recent text \cite{braun2025conditional}, and in the context of online convex optimization, the text \cite{hazan2016oco}.
Our starting point is the first linearly converging projection-free method, using the local linear oracle and shrinking-neighborhood
argument of Garber and Hazan \cite{garberhazan2016}.  Strong convexity
localizes the optimizer to a shrinking neighborhood, and a local
linear-optimization oracle then yields geometric convergence, which they obtained for polytopes.   The required local problem is
\[
  \min_{z\in\K\cap\B(x,s)}\ip{\nabla f(x)}z.
\]
Recent work rediscovers the local linear-minimization method and applies it to more general settings
\cite{richtarik2026}.  Proposition~\ref{prop:transfer} gives the
Garber--Hazan shrinking-neighborhood argument in the additive-error form
needed here; we include the proof for completeness.

Levy and Krause~\cite{levykrause2019} introduced the tangent-plane repair
that underlies our construction.  Given value and gradient access to a
smooth convex function defining $\K$, they locate a boundary point, project
a tentative gradient step onto its supporting hyperplane, and move inward
along the normal until feasibility is restored.  Their key geometric lemma
shows that this repair is quadratic in the tangential displacement, and
they use it to obtain optimal $O(\sqrt T)$ and $O(\log T)$ online regret
bounds.  We use the same geometric fact in a different way.  We minimize
the objective gradient over the intersection of a shrinking Euclidean ball
and one supporting halfspace, and then repair the resulting point using the
certified interior tangent ball.  This gives a local linear oracle with
quadratic additive error.  Combined with the Garber--Hazan
shrinking-neighborhood argument, it yields linear convergence for a fixed
smooth, strongly convex objective.

Liu and Grimmer's $\beta_{\mathrm{LG}}$-smooth sets are exactly our rolling-ball sets with
$\beta_{\mathrm{LG}}=1/\rho$ \cite{liugrimmer2025}.  They prove that the squared Minkowski
gauge is smooth and design projection-free radial methods using gauge and
boundary-normal computations.  With set smoothness alone their rate is
accelerated sublinear; their linear regime requires the set structure to be
both smooth and strongly convex.  In contrast, our set may have flat pieces. 
We also use their
gauge-smoothness theorem to recover the boundary normal from membership.

Gauge-based projection-free methods using membership oracles were developed in the pioneering work  by Mhammedi \cite{mhammedi2022}.  His reduction maps iterates from a
Euclidean ball to the feasible set through approximate gauge projections,
implemented using membership queries, and obtains optimal $\widetilde O(\sqrt T) / \widetilde O(\log T)$ regret for convex  / strongly convex losses over general bounded convex sets.  

More general membership-oracle reductions and cutting-plane methods are
given in \cite{leesidfordvempala2018}; these apply more broadly but have
different algorithms and oracle-complexity guarantees.

Hom-PGD obtains linear convergence conditional on an explicit homeomorphism
whose Jacobian is everywhere nonsingular and Lipschitz and whose
bi-Lipschitz constants are controlled \cite{hompgd2025}.  Such a map is not
provided by a membership oracle for a general convex body; in particular, the
proposed radial gauge map is generally not differentiable at the origin
unless the body is a Euclidean ball.  Thus its assumptions do not give the
membership-oracle result proved here.

\section{Idealized Algorithm and Main Result}\label{sec:local}

We say that a convex body $\K$ is \emph{$\rho$-smooth} if  every supporting hyperplane is tangent to an interior ball of radius
$\rho$.
Formally, for every
$q\in\bd\K$ and every outward supporting unit normal $n$ at $q$, that is,
every unit vector satisfying
\[
  \ip{n}{z-q}\le0
  \qquad\text{for all }z\in\K,
\]
we have
\begin{equation}\label{eq:rolling-ball}
\B(q-\rho n,\rho)\subseteq\K.
\end{equation}

We first describe the essence of the geometric optimization result assuming
\emph{exact} tangent access. Given a feasible point $x\in\K$ and a query
point $y\in\R^d$, one exact tangent query returns the last point
$q\in[x,y]\cap\K$. If $q=y$, this certifies that $y\in\K$; otherwise, the
query also returns an outward supporting unit normal $n$ at $q$.
Section~\ref{sec:membership}
implements this access to the required precision using membership queries
alone.
For any objective $f$ considered below, let
\[
  x^\star\in\arg\min_{x\in\K} f(x),
  \qquad
  f^\star:=f(x^\star).
\]

\begin{theorem}[Linear convergence from one tangent]\label{thm:geometric}
Let $\K\subset\R^d$ be a compact $\rho$-smooth convex body with
$\diam(\K)\le D$, and let $f$ be $\beta$-smooth and $\alpha$-strongly convex
on $\K$.  For an initial point $x_0\in\K$, suppose that
\[
   f(x_0)-f^\star \leq H_0\,
  \qquad
  \sup_{x\in\K}\norm{\nabla f(x)} \leq G,
\]
and define
\[
  \gamma:=\frac{\alpha}{\beta},
  \qquad
  \Gamma:=\frac{G}{\alpha\rho}.
\]
With exact tangent access, the shrinking-radius one-tangent method described
below returns $x_T\in\K$ satisfying 
$  f(x_T)-f^\star\le\eps$ after
\begin{equation}\label{eq:geometric-rate}
  T=O\!\left(
    (\gamma^{-1}+\Gamma)\log\frac{H_0}{\eps}
  \right)
\end{equation}
iterations, for every $0<\eps<H_0$.  Each iteration uses one objective
gradient and at most one exact tangent query.  
\end{theorem}

\paragraph{Initialization.}
The theorem may be used with any valid upper bounds $H_0$ and $G$.  For example, if $\nabla f(x_0) \neq 0$, then $  H_0:= D\norm{ \nabla f(x_0) },  G:=\norm{\nabla f(x_0)}+\beta D $ are valid.

The proof of this theorem is based on an outer reduction to local optimization, based on local linear oracles. The key is how to implement these oracles efficiently, which we do via the tangent oracle and using the smoothness of the set. 

\subsection{The outer reduction}
For $x\in\K$, radius $s>0$, and direction $c\in\R^d$, consider the Garber-Hazan local linear oracle problem
\begin{equation}\label{eq:localcap}
  \min_{z\in\K\cap\B(x,s)}\ip{c}{z}.
\end{equation}
The outer algorithm needs only an additive approximation whose error is
quadratic in $s$.  Fix an accuracy constant $a\ge0$.

\begin{algorithmbox}{Quadratic-accuracy local oracle}
Given $x\in\K$, $0<s\le  D$, and $c\in\R^d$, return
$p\in\K\cap\B(x,s)$ such that
\begin{equation}\label{eq:abstract-local-oracle}
  \ip{c}{p}
  \le
  \min_{z\in\K\cap\B(x,s)}\ip{c}{z}
  +a\norm{c}s^2.
\end{equation}
\end{algorithmbox}

The following proposition, due to Garber--Hazan  \cite{garberhazan2016}, is the shrinking-neighborhood argument that gives the outer reduction in the additive-error
form needed here. The schedule and proof appear in Appendix~\ref{app:outer-proof} for completeness. 

\begin{proposition}[Additive local-oracle transfer]\label{prop:transfer}
Assume $\diam(\K)\le D$, and let $f$ be $\beta$-smooth and
$\alpha$-strongly convex on $\K$, with
$\sup_{x\in\K}\norm{\nabla f(x)}\le G$. Suppose that, for some $a\ge0$,
\eqref{eq:abstract-local-oracle} is available for every $x\in\K$,
$0<s\le D$, and $c\in\R^d$. Given
$x_0\in\K$ and $H_0\ge f(x_0)-f^\star$, a feasible shrinking-radius
scheme using one gradient and one local-oracle call per iteration reaches
error at most $\eps$, for every $0<\eps<H_0$, in number of iterations bounded by 
\begin{equation}\label{eq:abstract-rate}
  O\!\left(\left(\frac{\beta}{\alpha}+\frac{aG}{\alpha}\right)
  \log\frac{H_0}{\eps}\right). 
\end{equation}
\end{proposition}

We now proceed to describe how to implement the local oracle, the main technical point. 

\subsection{The idealized one-tangent step}

\begin{algorithmbox}{One-tangent local oracle}
Given $x\in\K$, $0<s\le D$, and $0 \neq c \in\R^d$ (if $c=0$, return $x$):
\begin{enumerate}
  \item Query the endpoint $x-s\frac{c}{\norm c}$ and write the returned
  point as
  \[
    q=x-\tau\frac{c}{\norm c},
    \qquad 0\le\tau\le s.
  \]
  \item If $\tau=s$, return $q$. Otherwise, let $n$ be the returned outward
  unit normal at $q$.
  \item Compute
  \begin{equation}\label{eq:exactrelax}
    p\in\arg\min\{\ip{c}{z}:z\in\B(x,s),\ \ip{n}{z-q}\le0\}.
  \end{equation}
  \item Return
  \[
    \widehat p
    =\operatorname{proj}_{\B(q-\rho n,\rho)}(p).
  \]
\end{enumerate}
\end{algorithmbox}

Write $\B_q:=\B(q-\rho n,\rho)$.  Step~3 has a closed-form solution.
For a unit vector $n$ and a halfspace $\ip n z\le b_0$, set
$b=b_0-\ip n x$ and $c_\perp=c-\ip c n\,n$.  If the unconstrained ball
minimizer is feasible, use it; otherwise
\[
  p=x+bn-\sqrt{s^2-b^2}\,\frac{c_\perp}{\norm{c_\perp}},
\]
with the last term defined to be zero when $c_\perp=0$.

\begin{figure}[H]
\centering
\begin{tikzpicture}[
    scale=1.25,
    font=\small,
    >=Latex
]
\coordinate (x)    at (-1.30,-0.80);
\coordinate (q)    at (0,0);
\coordinate (p)    at (1,0);
\coordinate (m)    at (0,-1);
\coordinate (phat) at (0.707,-0.293);

\path[fill=setblue!10] (0,-2) circle[radius=2];
\draw[setblue,very thick] (0,-2) circle[radius=2];
\node[setblue] at (-1.35,-2.15) {$\K$};

\draw[gray!70,dashed,thick] (x) circle[radius=2.435];
\node[gray!75] at (-2.55,0.65) {$\B(x,s)$};

\draw[gray!70,dashed,thick] (-2.65,0) -- (1.75,0);
\node[gray!70,above] at (-1.55,0)
    {$\{z:\ip{n}{z-q}=0\}$};

\path[fill=repairorange!12] (m) circle[radius=1];
\draw[repairorange,very thick] (m) circle[radius=1];
\node[repairorange!90!black] at (0.55,-1.35) {$\B_q$};

\draw[repairorange,thick]
    (m) -- (q)
    node[midway,right] {$\rho$};

\draw[setblue,very thick,->]
    (q) -- (0,0.75)
    node[right] {$n$};

\draw[blue!70!black,very thick,->]
    (x) -- (q)
    node[midway,above left=-1pt] {$q-x$};

\draw[red!75!black,very thick,->]
    (q) -- (p)
    node[midway,above=2pt] {$p-q$};

\draw[repairorange,very thick,->]
    (p) -- (phat)
    node[midway,right=2pt] {projection};

\fill[black] (x) circle[radius=1.7pt]
    node[below left] {$x$};

\fill[black] (q) circle[radius=1.7pt]
    node[above left=2pt] {$q$};

\draw[red!75!black,fill=white,very thick]
    (p) circle[radius=2pt]
    node[above right] {$p$};

\fill[repairorange] (m) circle[radius=1.4pt]
    node[below left] {$q-\rho n$};

\fill[querygreen] (phat) circle[radius=2pt];
\node[querygreen!70!black,below right] at (phat)
    {$\widehat p$};

\end{tikzpicture}

\caption{Geometry of Lemma~\ref{lem:exact-tangent}.
The boundary point $q$ lies on the ray from $x$ in direction $-c$, so
$q-x$ is a nonnegative multiple of $-c$.  The relaxed minimizer $p$
lies in the supporting hyperplane, and hence
$\ip{n}{p-q}=0$.  Projecting $p$ onto the certified interior tangent
ball $\B_q=\B(q-\rho n,\rho)$ produces $\widehat p$.  The projection
moves $p$ inward by only $O(\norm{p-q}^2/\rho)$ while keeping
$\widehat p$ inside $\B(x,s)$.}
\label{fig:idealized-oracle}
\end{figure}
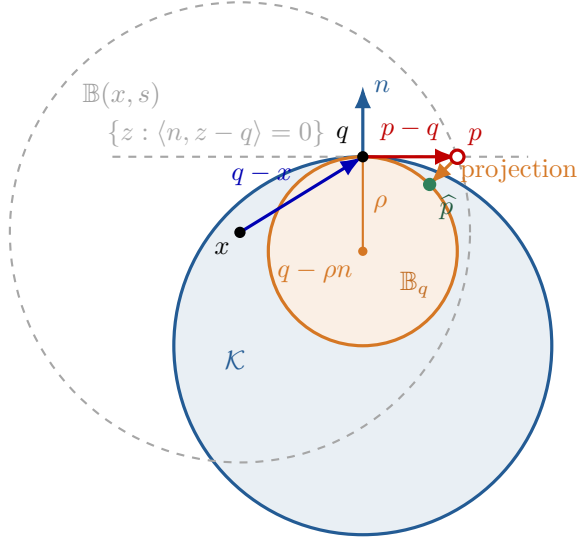

\begin{lemma}[Exact tangent accuracy]\label{lem:exact-tangent}
The point $\widehat p$ belongs to $\K\cap\B(x,s)$ and satisfies
\begin{equation}\label{eq:exact-error}
  \ip{c}{\widehat p}
  \le
  \min_{z\in\K\cap\B(x,s)}\ip{c}{z}
  +\frac{2}{\rho}\norm{c}s^2.
\end{equation}
\end{lemma}

\begin{proof}
The true local cap is feasible for \eqref{eq:exactrelax}, so
\begin{equation}\label{eq:p-relax-opt}
  \ip{c}{p}\le
  \min_{z\in\K\cap\B(x,s)}\ip{c}{z}.
\end{equation}
We now go through a sequence of geometric steps as follows. 

\paragraph{The halfspace constraint is active at $p$.}
Let $d_c:=-c/\norm c$, so $q=x+\tau d_c$ with $\tau<s$.
Since $q$ is the last feasible point in direction $d_c$, clearly
$\ip n{d_c}\ge0$.  Hence the unconstrained ball minimizer $x+sd_c$
satisfies
\[
  \ip n{x+sd_c-q}=(s-\tau)\ip n{d_c}\ge0.
\]
It therefore lies outside or on the supporting halfspace, so the halfspace
constraint in \eqref{eq:exactrelax} is active:
\begin{equation}\label{eq:tangent-p}
  \ip n{p-q}=0.
\end{equation}

\paragraph{The repair distance.}
Both $p$ and $q$ lie in $\B(x,s)$, so $\norm{p-q}\le2s$.
Using the elementary inequality when $a>0$
\[
  \sqrt{a^2+b^2}-a
  =\frac{b^2}{\sqrt{a^2+b^2}+a}
  \le\frac{b^2}{2a},
\]
we obtain
\begin{align}
  \dist(p,\B_q)
  &=\sqrt{\rho^2+\norm{p-q}^2}-\rho \notag\\
  &\le\frac{\norm{p-q}^2}{2\rho}
  \le\frac{2s^2}{\rho}.
\label{eq:exact-distance}
\end{align}

\paragraph{Preservation of the local ball.}
Optimality of $p$ against the feasible point $q$, together with the preceding
paragraph, gives
\[
  \ip{q-x}{p-q}\ge0,
  \qquad
  \ip{q-x}{n}\ge0.
\]
The repair vector $\widehat p-p$ is a nonpositive multiple of
\[
  p-(q-\rho n)=(p-q)+\rho n,
\]
and therefore
\[
  \ip{q-x}{\widehat p-p}\le0.
\]
Also, since $q\in\B_q$, projection cannot increase the distance to $q$:
\[
  \norm{\widehat p-q}\le\norm{p-q}.
\]
Consequently,
\[
  \norm{\widehat p-x}^2-\norm{p-x}^2
  =
  \norm{\widehat p-q}^2-\norm{p-q}^2
  +2\ip{q-x}{\widehat p-p}
  \le0.
\]
Thus $\widehat p\in\B(x,s)$.  Since also
$\widehat p\in\B_q\subseteq\K$, we conclude that
$\widehat p\in\K\cap\B(x,s)$.

 \paragraph{Objective accuracy.}
Finally,
\begin{align*}
  \ip c{\widehat p}
  &\le \ip c p+\norm c\,\norm{\widehat p-p}\\
  &\le
  \min_{z\in\K\cap\B(x,s)}\ip c z
  +\frac{2}{\rho}\norm c\,s^2,
\end{align*}
where the second inequality follows from
\eqref{eq:p-relax-opt} and \eqref{eq:exact-distance}.
This proves \eqref{eq:exact-error}.
 
\end{proof}

\begin{proof}[Proof of Theorem~\ref{thm:geometric}]
Consider any local call $(x,s,c)$. If $c=0$, return $x$. If $\tau=s$, then
$q=x-sc/\norm c$ minimizes the linear objective over the whole ball
$\B(x,s)$ and is an exact solution of \eqref{eq:localcap}. Otherwise,
Lemma~\ref{lem:exact-tangent} implements
\eqref{eq:abstract-local-oracle} with $a=2/\rho$.  Applying
Proposition~\ref{prop:transfer} proves \eqref{eq:geometric-rate}.
Feasibility and the stated oracle usage follow directly from
the one-tangent construction.
\end{proof}

\section{Implementation from membership queries}\label{sec:membership}
To implement the method using membership queries, decrease $r$ if necessary
and assume that known radii $0<r\le\rho\le R$ satisfy
\begin{equation}\label{eq:sandwich}
  \B(0,r)\subseteq\K\subseteq\B(0,R).
\end{equation}
Since $\diam(\K)\le2R$, Theorem~\ref{thm:geometric} applies with diameter
bound $2R$.  Assume access to the exact oracle
$\MEM(z)=\mathbf 1\{z\in\K\}$.

The reduction has two ingredients.  First, membership bisection locates the
last feasible point $q$ on the segment from $x$ toward $x-sc/\norm c$.
Second, membership also gives arbitrarily accurate values of the Minkowski
gauge.  For $w\in\R^d$, define
\[
  p_{\K}(w):=\inf\{t>0:w\in t\K\}.
\]
For every $t>0$,
\[
  p_{\K}(w)\le t
  \quad\Longleftrightarrow\quad
  \MEM(w/t)=1,
\]
so gauge values can be computed by bisection.  The squared gauge is smooth
under the rolling-ball condition \cite{liugrimmer2025}, and its normalized
gradient at $q$ is the outward unit normal.  Centered finite differences
therefore recover the normal using $2d$ gauge evaluations.
Appendix~\ref{app:membership-details} gives the details and shows how to
repair the approximate tangent while preserving exact feasibility.

\begin{proposition}[Membership implementation of the local oracle]
\label{prop:membership-local}
Let $\K$ be a $\rho$-smooth convex body satisfying the ball
sandwich~\eqref{eq:sandwich}.  For every $x\in\K$, $0<s\le2R$, and
$c\in\R^d$, there is a deterministic algorithm using queries to the exact
membership oracle $\MEM$ that returns
$\widehat p\in\K\cap\B(x,s)$ satisfying
\[
  \ip c{\widehat p}
  \le
  \min_{z\in\K\cap\B(x,s)}\ip c z
  +\frac{9}{\rho}\norm c\,s^2.
\]
The number of membership queries is
\[
  O\!\left(
    d\log\left(d+\frac Rr+\frac Rs\right)
  \right).
\]
\end{proposition}

\begin{corollary}[Membership-oracle implementation]\label{cor:membership}
Under the assumptions of Theorem~\ref{thm:geometric}, additionally suppose
that the ball sandwich~\eqref{eq:sandwich} is known and that an exact
membership oracle for $\K$ is available.  Writing
\[
  \gamma:=\frac{\alpha}{\beta},
  \qquad
  \Gamma:=\frac{G}{\alpha\rho},
\]
the shrinking-radius method reaches objective error at most $\eps$, for every
$0<\eps<H_0$, in
\begin{equation}\label{eq:Tbound}
  T=O\!\left(
    (\gamma^{-1}+\Gamma)\log\frac{H_0}{\eps}
  \right)
\end{equation}
iterations.  Each iteration uses one objective gradient and
$O(d\log Q_\eps)$ exact membership queries, where
\begin{equation}\label{eq:Qeps}
  Q_\eps=
  d+\frac Rr+
  \frac{\alpha R^2(1+\Gamma)^2}{\eps}.
\end{equation}
Hence the total number of membership queries is $O(dT\log Q_\eps)$.
The method uses no objective values, no projection onto $\K$, and no local
or global linear-optimization oracle.
\end{corollary}

\begin{proof}
Proposition~\ref{prop:membership-local} implements
\eqref{eq:abstract-local-oracle} with $a=9/\rho$, so
Proposition~\ref{prop:transfer} gives \eqref{eq:Tbound}.  Before termination,
the radius schedule in Appendix~\ref{app:outer-proof} satisfies
\[
  \frac{R^2}{s_t^2}
  =O\!\left(
    1+\frac{\alpha R^2(1+\Gamma)^2}{\eps}
  \right).
\]
Substitution in the query bound of
Proposition~\ref{prop:membership-local} gives
$O(d\log Q_\eps)$ membership queries per iteration.
\end{proof}

\section{Conclusion}

We describe a simple projection-free linearly converging algorithm whose main primitive is tangent computation, and relies on the underlying set being smooth. 

The one-tangent argument relies essentially on smoothness of the feasible
set.  At a corner, a single tangent can differ from the set by $\Theta(s)$
rather than $O(s^2)$, and our argument no longer applies.
Approximate membership, adaptive parameter choices, and multi-tangent repairs
for nonsmooth sets are natural further directions.

\bibliographystyle{plain}
\bibliography{main}

\appendix

\section{The additive local-oracle reduction}
\label{app:outer-proof}

\begin{proof}[Proof of Proposition~\ref{prop:transfer}]
Write
\[
  \kappa:=\frac{\beta}{\alpha},
  \qquad
  b:=\frac{aG}{\alpha},
\]
and set
\[
  \theta:=\frac{1}{4(1+b)},
  \qquad
  \eta:=\frac{1+b}{1+b+\kappa},
  \qquad
  \sigma:=\frac{1}{16(1+b+\kappa)}
          =\frac{\eta\theta}{4}.
\]
Let
\[
  h_t:=f(x_t)-f^\star,
  \qquad
  \Delta_t:=(1-\sigma)^tH_0,
\]
and choose
\[
  s_t:=\min\left\{
    D,\,
    \theta\sqrt{\frac{2\Delta_t}{\alpha}}
  \right\}.
\]
Given \(x_t\), call the local oracle with
\((x_t,s_t,\nabla f(x_t))\), obtaining \(p_t\), and update
\[
  x_{t+1}=x_t+\eta(p_t-x_t).
\]

We prove by induction that \(h_t\le\Delta_t\). Suppose this holds at
time \(t\). Strong convexity and constrained optimality of \(x^\star\)
give
\[
  \norm{x_t-x^\star}
  \le\sqrt{\frac{2h_t}{\alpha}}
  \le\sqrt{\frac{2\Delta_t}{\alpha}}.
\]
For \(x_t\ne x^\star\), let
\[
  \lambda_t:=
  \min\left\{1,\frac{s_t}{\norm{x_t-x^\star}}\right\},
\]
and take \(\lambda_t=1\) when \(x_t=x^\star\). Since
\(D\ge\norm{x_t-x^\star}\), we have \(\lambda_t\ge\theta\). Thus
\[
  y_t:=x_t+\lambda_t(x^\star-x_t)
  \in\K\cap\B(x_t,s_t).
\]

Writing \(g_t=\nabla f(x_t)\), the local-oracle guarantee and convexity
give
\begin{align*}
  \ip{g_t}{p_t-x_t}
  &\le
  \lambda_t\ip{g_t}{x^\star-x_t}
  +a\norm{g_t}s_t^2\\
  &\le
  -\theta h_t+2b\theta^2\Delta_t.
\end{align*}
Consequently, smoothness implies
\begin{align*}
  h_{t+1}
  &\le
  h_t+\eta\ip{g_t}{p_t-x_t}
  +\frac{\beta\eta^2}{2}s_t^2\\
  &\le
  (1-\eta\theta)h_t
  +\left(2\eta b\theta^2
    +\kappa\eta^2\theta^2\right)\Delta_t\\
  &\le
  \left[
    1-\eta\theta
    \left(1-2b\theta-\kappa\eta\theta\right)
  \right]\Delta_t\\
  &\le
  (1-\sigma)\Delta_t
  =\Delta_{t+1},
\end{align*}
where the last inequality uses
\[
  2b\theta\le\frac12,
  \qquad
  \kappa\eta\theta\le\frac14.
\]
This proves the induction. The iterates are feasible because
\(x_t,p_t\in\K\) and \(0<\eta\le1\). Finally,
\[
  h_T\le(1-\sigma)^T H_0
  \le e^{-\sigma T}H_0,
\]
and
\[
  \frac1{\sigma}
  =16(1+b+\kappa)
  =O\left(
    \frac{\beta}{\alpha}+\frac{aG}{\alpha}
  \right).
\]
The claimed iteration bound follows.
\end{proof}

\section{Membership-oracle details}\label{app:membership-details}

Let $\psi(w):=p_{\K}(w)^2/2$, where $p_{\K}$ is the Minkowski gauge
defined in Section~\ref{sec:membership}. Liu and
Grimmer~\cite{liugrimmer2025} show that
\[
  \operatorname{Lip}(\nabla\psi)
  \le L_\psi:=\frac{r+R^2/\rho}{r^3},
  \qquad
  \nabla\psi(q)=\frac{n(q)}{\ip{n(q)}q},
  \qquad
  r\le\ip{n(q)}q\le R
\]
for every $q\in\bd\K$, where $n(q)$ is the outward unit normal. Thus
$\norm{\nabla\psi(q)}\ge1/R$ and its normalized gradient is $n(q)$.
Moreover,
\[
  p_{\K}(w)\le t
  \quad\Longleftrightarrow\quad
  \MEM(w/t)=1,
\]
so membership bisection evaluates $\psi$ to arbitrary accuracy.

\begin{lemma}[Tangent approximation from membership]
\label{lem:membership-tangent}
Let $x\in\K$, let $y\notin\K$ satisfy $\norm{y-x}\le2R$, and let $q$ be
the last feasible point on $[x,y]$. For every $0<\delta\le r$, there is
a deterministic algorithm using queries to the exact membership oracle
$\MEM$ that returns a feasible $\widehat q\in[x,q]$ and a unit vector
$\widetilde n$ satisfying
\[
  \norm{\widehat q-q}\le\delta,
  \qquad
  \norm{\widetilde n-n(q)}
  \le\frac{\delta}{2R+\rho}.
\]
It uses
\[
  O\!\left(d\log\left(d+\frac Rr+\frac R\delta\right)\right)
\]
membership queries.
\end{lemma}

\begin{proof}
Set $\zeta:=\delta/(2R+\rho)$. Bisect $[x,y]$ to find a feasible
$\widehat q\in[x,q]$ with
\[
  \norm{\widehat q-q}
  \le\min\left\{\delta,\frac{\zeta}{8RL_\psi}\right\}.
\]
Approximate each $\psi(\widehat q\pm he_i)$ to accuracy $\tau$, where
\[
  h:=\frac{\zeta}{8RL_\psi\sqrt d},
  \qquad
  \tau:=\frac{\zeta h}{16R\sqrt d}
\]
and let $\widetilde g$ be the centered finite-difference vector. Then
\[
  \norm{\widetilde g-\nabla\psi(q)}
  \le\sqrt d\left(\frac{L_\psi h}{2}+\frac{\tau}{h}\right)
    +L_\psi\norm{\widehat q-q}
  \le\frac{\zeta}{4R}.
\]
Since $\norm{\nabla\psi(q)}\ge1/R$, normalizing $\widetilde g$ gives
the claimed $\widetilde n$. There are $2d$ gauge evaluations and one
segment bisection. Since $r\le\rho\le R$,
$R^2L_\psi\le(R/r)^2+(R/r)^4$ and
$1/\zeta\le3R/\delta$. Thus all required accuracies are
inverse-polynomial in $d+R/r+R/\delta$, proving the query bound.
\end{proof}

\begin{proof}[Proof of Proposition~\ref{prop:membership-local}]
Return $x$ if $c=0$, and otherwise set
$y=x-sc/\norm c$. If $y\in\K$, return $y$. Otherwise, let $q$ be the
last feasible point on $[x,y]$ and set
\[
  \delta:=\frac14\min\left\{r,\frac{s^2}{\rho}\right\}.
\]
Run the algorithm of Lemma~\ref{lem:membership-tangent} to obtain
$\widehat q,\widetilde n$, and define
\[
  \B_q:=\B(q-\rho n(q),\rho),\qquad
  \widetilde H:=
  \{z:\ip{\widetilde n}{z-\widehat q}\le2\delta\},\qquad
  \widetilde\B:=
  \B(\widehat q-\rho\widetilde n,\rho-2\delta).
\]
Let
\[
  p\in\arg\min_{z\in\B(x,s)\cap\widetilde H}\ip c z,
  \qquad
  u:=\operatorname{proj}_{\widetilde\B}(p),
\]
and let $\widehat p$ be the projection of $u$ onto $\B(x,s)$ along the
segment from $x$ to $u$.

Write $n=n(q)$. For every $z\in\K$,
\[
  \ip{\widetilde n}{z-\widehat q}\le2\delta,
  \qquad
  \norm{(\widehat q-\rho\widetilde n)-(q-\rho n)}
  \le2\delta.
\]
Hence $\K\subseteq\widetilde H$ and
$\widetilde\B\subseteq\B_q\subseteq\K$, so
\[
  \ip c p\le
  \min_{z\in\K\cap\B(x,s)}\ip c z.
\]

Let $v=p-q$. Since $\norm v\le2s$ and $p\in\widetilde H$,
\[
  \ip n v
  \le3\delta+\frac{2s}{2R+\rho}\delta
  \le5\delta.
\]
The elementary inequality
\[
  \dist(q+v,\B_q)
  \le[\ip n v]_++\frac{\norm v^2}{2\rho}
\]
therefore gives
\[
  \norm{u-p}
  =\dist(p,\widetilde\B)
  \le\frac{2s^2}{\rho}+9\delta.
\]
Since $p\in\B(x,s)$, the final radial projection moves $u$ by at most
$\norm{u-p}$. Thus
\[
  \norm{\widehat p-p}
  \le2\norm{u-p}
  \le\frac{4s^2}{\rho}+18\delta
  \le\frac{9s^2}{\rho}.
\]
Moreover, $x,u\in\K$ and $\widehat p\in[x,u]$, so
$\widehat p\in\K\cap\B(x,s)$. It follows that
\[
  \ip c{\widehat p}
  \le
  \min_{z\in\K\cap\B(x,s)}\ip c z
  +
  \frac9\rho\norm c\,s^2.
\]

Finally, $R/\delta=O(R/r+(R/s)^2)$, and
Lemma~\ref{lem:membership-tangent} gives the claimed
$O(d\log(d+R/r+R/s))$ query bound.
\end{proof}

\end{document}